\documentclass{article}
\usepackage{amsmath,amsxtra,amssymb,amsthm,amsfonts}
\usepackage{graphicx} 
\usepackage{color,graphicx}
\usepackage[margin=1.05in]{geometry}
\usepackage[dvipsnames]{xcolor}
\usepackage{url}

\definecolor{myHexColor}{HTML}{FFD700} 

\usepackage[pdftex,linktocpage=true,colorlinks,citecolor=myHexColor,linkcolor=pink,pagebackref]{hyperref}
\usepackage{hyperref}
\usepackage{cleveref}
\usepackage[alphabetic,backrefs]{amsrefs}
\usepackage{caption}
\usepackage{subcaption}
\usepackage[normalem]{ulem}

\usepackage{pgfplots}
\usepackage{pgf}
\usepackage{tikz}
\usetikzlibrary{patterns}
\usetikzlibrary{arrows.meta}
\usepgfplotslibrary{patchplots} 
\usetikzlibrary{pgfplots.patchplots} 
\pgfplotsset{width=9cm,compat=1.5.1}

\def\C{\mathbb{C}}

\def\R{\mathbb{R}}

\def\Z{\mathbb{Z}}

\def\0{\mathbf{0}}

\numberwithin{equation}{section}
\newtheorem{lemma}{Lemma}[section]

\newtheorem{theorem}[lemma]{Theorem}

\begin{document}

\title{
     The cycle $C_9$ does not admit uniform mixing
}

\author{  
  Alison Gray\textsuperscript{1}, Pransu Patel\textsuperscript{1}, and Isaiah Young\textsuperscript{2}
}

\maketitle

\begin{abstract}
   We study continuous-time quantum walks (CTQWs) on cycles. In particular, we prove that the cycle $C_9$ does not admit uniform mixing at any time via algebraic geometry and Gröbner basis techniques to rule out all cyclic 9-roots. 

    \medskip

    \medskip
	
\end{abstract}

\addtocounter{footnote}{1}
\footnotetext{Department of Mathematics \& Computer Science, Brandon University, Brandon, MB, Canada R7A 6A9}
\addtocounter{footnote}{1}
\footnotetext{Department of Mathematics \& Computer Science, Mount Allison University, Sackville, NB, Canada E4L 1E4}

\section{Introduction}\label{sec:intro}

Continuous time quantum walks have become an important topic in quantum information theory because they have many applications in areas such as quantum computing and quantum communication. One interesting property of a quantum walk is uniform mixing, which occurs when every vertex of a graph has the same probability at a particular time. More formally, given a graph $G$ on $n$ vertices, uniform mixing occurs when every entry of the transition matrix $U(t) := e^{itA(G)}$ has a squared magnitude of $1/n$ at a particular time $t \in \R$. This property was first studied in 2002 by Moore and Russell in \cite{moore2002quantum}, and since then there has been extensive research conducted on the mixing properties of numerous families of graphs. While uniform mixing is known for multiple graph families, many important cases still remain open. 

Several graph families have been shown to admit uniform mixing, including the complete graphs $K_2$, $K_3$, and $K_4$ \cite{Ahmadi2002Mixing}, the Binary Hamming graphs $H(d,2)$ (hypercubes) \cite{mullin2013uniform}, folded Hamming graphs over $H(d,3)/\langle\mathbf{1}\rangle$ and $H(d,4)/\langle\mathbf{1}\rangle$ under certain conditions \cite{mullin2013uniform}, the cycle graphs $C_3$ and $C_4$ \cite{Ahmadi2002Mixing}, the Paley graph of order $9$ \cite{godsil2017uniform}, and the Star graph $K_{1,3}$ \cite{godsil2017uniformcayley}. Cycle graphs are one of the simplest graph families for studying uniform mixing.

It is also important to note that no even cycle $C_n$ with $n>4$ and no odd prime cycle $C_n$ with $n>3$ admit uniform mixing \cite{mullin2013uniform}.
In her thesis, Mullin \cite{mullin2013uniform} identified the cycle graph $C_9$ as one of the remaining open problems for uniform mixing. Since $9$ is the first odd composite number, the methods used to study odd prime cycles do not apply, making $C_9$ the first unresolved odd composite cycle.
Meanwhile, in the context of computational algebraic geometry, Sabeti \cite{sabeti2015scheme} showed that there are six two-dimensional families of cyclic 9-roots together with $6642$ isolated cyclic 9-roots.

In this paper, we study  these remaining families of cyclic 9-roots by converting them into cyclic Type-II matrices. We then compare these matrices with the transition matrix of $C_9$ obtained from the Fourier diagonalization of the adjacency matrix. This comparison shows that certain conditions must hold if uniform mixing were to occur. We show that these conditions lead to a contradiction and therefore conclude that $C_9$ does not admit uniform mixing at any time.

\textbf{Authors' Note:} The results in this paper were proven independently, yet concurrently with \cite{cao2026uniform}. While their work involves constructing a new Gröbner basis and showing that there is no time $t$ at which $U(t)$ is flat based on the given parameters, our work is based on an already existing Gröbner basis and comparing the Type-II matrix to the Fourier-diagonalized matrix which led us to a contradiction. Therefore, there is little overlap between our works.

\section{Proof of $C_9$}\label{sec:proof}
In this section, $\zeta = \frac{1-3i}{2}$ is the primitive sixth root of unity, which is the same as $-\omega$ (where $\omega = \frac{-1+3i}{2}$ is the primitive third root of unity). Recall that a complex Hadamard matrix is an $n\times n$ matrix $H$ whose entries have unit modulus and satisfy
\[
HH^{*}=nI,
\]
where $H^{*}$ denotes the conjugate transpose of $H$. Motivated by Jones' work on spin models \cite{jones1989knot}, a \emph{Type-II matrix} is an $n\times n$ matrix $W$ with no zero entries satisfying
\[
W\circ W^{(-)T}=J
\]
and
\[
WW^{(-)T}=nI,
\]
where $W^{(-)}$ denotes the Schur inverse of $W$, $\circ$ denotes the Schur product, and $I$ and $J$ denote the identity and all-ones matrices, respectively.


We will work with all the cyclic 9-root solutions that can generate complex Hadamard matrices of order 9. This will help us to compare the transition matrix with the Type-II matrix that we form using \cite[Lemma 4.7.1]{mullin2013uniform}. The first row of said Type-II matrix can be formed by defining the first element $x_0 = 1$ and rest of the elements of the first row as $x_j = \prod_{i=0}^{j-1} z_i$ where each $z_i$ is a cyclic $n$-root for $1\leq j \leq n-1$. To do so we will use the ideals formed in \cite{sabeti2015scheme} to get the solution of the cyclic 9-roots.

In \cite[Eq. 13]{sabeti2015scheme}, we can reduce the equations to Gröbner basis as shown in \cite[Appendix]{sabeti2015scheme}. Here, we will first re-write them with our notation to not cause confusion.

Following Sabeti's notation, $C_2^{9,i}$ denotes the $i^{th}$ two-dimensional prime ideal in the minimal prime decomposition of the radical of the cyclic 9-roots ideal, $\sqrt{I(C_9)}$, where the superscript 9 refers to the cyclic 9-roots system, the subscript 2 indicates the dimension of the corresponding irreducible component, and $i \in {1,...,6}$ indexes the six distinct two-dimensional components.

\begin{align} 
    C_2^{9,1} &=  I(x_1 - \omega^2x_7,x_1-\omega x_4,x_2-\omega^2x_8, x_2-\omega x_5, x_3 - \omega^2x_9,x_3 -\omega x_6, x_1x_2x_3 - \omega) \\
    C_2^{9,2} &=  I(x_1 - \omega^2x_7,x_1-\omega x_4,x_2-\omega^2x_8, x_2-\omega x_5, x_3 - \omega^2x_9,x_3 -\omega x_6, x_1x_2x_3 - \omega^2) \\
    C_2^{9,3} &= I(x_1 - \omega^2x_7,x_1-\omega x_4,x_2-\omega^2x_8, x_2-\omega x_5, x_3 - \omega^2x_9,x_3 -\omega x_6, x_1x_2x_3 - 1) \\
    C_2^{9,4} &=  I(x_1 - \omega x_7,x_1-\omega^2 x_4,x_2-\omega x_8, x_2-\omega^2 x_5, x_3 - \omega x_9,x_3 -\omega^2 x_6, x_1x_2x_3 - \omega) \\
    C_2^{9,5} &=  I(x_1 - \omega x_7,x_1-\omega^2 x_4,x_2-\omega x_8, x_2-\omega^2 x_5, x_3 - \omega x_9,x_3 -\omega^2 x_6, x_1x_2x_3 - \omega^2) \\
    C_2^{9,6} &=  I(x_1 - \omega x_7,x_1-\omega^2 x_4,x_2-\omega x_8, x_2-\omega^2 x_5, x_3 - \omega x_9,x_3 -\omega^2 x_6, x_1x_2x_3 - 1)
\end{align}

\begin{theorem}
    The cycle graph $C_9$ does not exhibit uniform mixing at any time $t \in \R$.
\end{theorem}
\begin{proof}
    We prove this by contradiction. To show a contradiction, we start with the assumption that $C_9$ exhibits uniform mixing at some time $t = \tau \in \R$. We denote $U_2(\tau)$ as the transition matrix derived from converting these cyclic 9-roots into a Type-II matrix. Let $x_1 = u$ and $x_2 = v$ (where $u,v \in \C$), and we get the cyclic 9-roots as follows:
    
    From $C_2^{9,1}$, 
    \begin{align*}
        x_1x_2x_3 - \omega &= 0 &\implies x_3 &= \frac{\omega}{uv} \\
        x_1 - \omega x_4 &= 0 &\implies x_4 &= \omega^2u \\
        x_2 - \omega x_5 &= 0 &\implies x_5 &= \omega^2v \\
        x_3 - \omega x_6 &= 0 &\implies x_6 &= \frac{1}{uv} \\
        x_1 - \omega^2x_7 &= 0 &\implies x_7 &= \omega u \\
        x_2 - \omega^2x_8 &= 0 &\implies x_8 &= \omega v \\
        x_3 - \omega^2x_9 &= 0 &\implies x_9 &= \frac{\omega^2}{uv}
    \end{align*}

    Similarly, for 2.2 to 2.6,
    
        \begin{center}
        \begin{tabular}{cccccc}
        - & $C_2^{9,2}$ & $C_2^{9,3}$ & $C_2^{9,4}$ & $C_2^{9,5}$ & $C_2^{9,6}$\\
            $x_1$ & $u$ & $u$ & $u$ & $u$ & $u$\\
            $x_2$ & $v$ & $v$ & $v$ & $v$ & $v$\\
            $x_3$ & $\frac{\omega^2}{uv}$ & $\frac{1}{uv}$ & $\frac{\omega}{uv}$ & $\frac{\omega^2}{uv}$ & $\frac{1}{uv}$\\
            $x_4$ & $\omega^2u$ & $\omega^2u$ & $\omega u$ & $\omega u$ & $\omega u$\\
            $x_5$ & $\omega^2v$ & $\omega^2v$ & $\omega v$ & $\omega v$ & $\omega v$\\
            $x_6$ & $\frac{\omega}{uv}$ & $\frac{\omega^2}{uv}$ & $\frac{\omega^2}{uv}$ & $\frac{1}{uv}$ & $\frac{\omega}{uv}$\\
            $x_7$ & $\omega u$ & $\omega u$ & $\omega^2u$ & $\omega^2u$ & $\omega^2u$\\
            $x_8$ & $\omega v$ & $\omega v$ & $\omega^2v$ & $\omega^2v$ & $\omega^2v$\\
            $x_9$ & $\frac{1}{uv}$ & $\frac{\omega}{uv}$ & $\frac{1}{uv}$ & $\frac{\omega}{uv}$ & $\frac{\omega^2}{uv}$\\
        \end{tabular}
        \label{tab:placeholder}
        \end{center}

    Next, we use \cite[Lemma 4.7.1]{mullin2013uniform} to convert the cyclic 9-roots into a 9 $\times$ 9 Type-II matrix. The following are the first row elements that we get from each of the ideals: \\

        \begin{center}
        \begin{tabular}{cccccccc}
            - & $C_2^{9,1}$ & $C_2^{9,2}$ & $C_2^{9,3}$ & $C_2^{9,4}$ & $C_2^{9,5}$ & $C_2^{9,6}$ & generalization\\
            $U_2(\tau)_{1,0}$ & $1$ & $1$ & $1$ & $1$ & $1$ & $1$ & $1$\\
            $U_2(\tau)_{1,1}$ & $u$ & $u$ & $u$ & $u$ & $u$ & $u$ & $u$\\
            $U_2(\tau)_{1,2}$ & $uv$ & $uv$ & $uv$ & $uv$ & $uv$ & $uv$ & $uv$\\
            $U_2(\tau)_{1,3}$ & $\omega$ & $\omega^2$ & $1$ & $\omega$ & $\omega^2$ & $1$ & $\rho$\\
            $U_2(\tau)_{1,4}$ & $u$ & $\omega u$ & $\omega^2u$ & $\omega^2u$ & $u$ & $\omega u$ & $\rho\sigma u$\\
            $U_2(\tau)_{1,5}$ & $\omega^2uv$ & $uv$ & $\omega uv$ & $uv$ & $\omega uv$ & $\omega^2uv$ & $\rho\sigma^2uv$\\
            $U_2(\tau)_{1,6}$ & $\omega^2$ & $\omega$ & $1$ & $\omega^2$ & $\omega$ & 1 & $\rho^2$\\
            $U_2(\tau)_{1,7}$ & $u$ & $\omega^2u$ & $\omega u$ & $\omega u$ & $u$ & $\omega^2u$ & $\rho^2\sigma^2u$\\
            $U_2(\tau)_{1,8}$ & $\omega uv$ & $uv$ & $\omega^2uv$ & $uv$ & $\omega^2uv$ & $\omega uv$ & $\rho^2\sigma uv$\\
        \end{tabular}
        \label{tab:placeholder}
        \end{center}
        \

    We generalize the 6 results in the last column where $\rho = \{1, \omega, \omega^2\}$ and $\sigma =\{\omega,\omega^2\}$. Now, we reduce the generalization by using the symmetry in $C_9$. From $U(t)$, we know that $a_8 = a_1$, $a_7 = a_2$, $a_6 = a_3$, and $a_5 = a_4$. 
    \\
    From $a_6 = a_3 \implies \rho^2 = \rho \implies \rho = 1$. \\
    From $a_7 = a_2 \implies \rho^2\sigma^2u = uv \implies \sigma^2 = v$. \\
    From $a_8 = a_1 \implies \rho^2\sigma uv = u \implies \sigma v = 1 \implies \sigma^3 = 1 \implies \sigma = \omega$. Therefore, $v = \sigma^2 = \omega^2$.

    Therefore, by substituting the values for $\sigma$, $\rho$, and $v$, we get the following as the first row of the Type-II matrix:
    
    \[ U_2(\tau) = \frac{1}{3}\begin{bmatrix}
        1 & u & \omega^2u & 1 & \omega u & \omega u & 1 & \omega^2u & u 
    \end{bmatrix} \]

    We know from \cite[Lemma 4.5.2]{mullin2013uniform}, $U(t) = \frac{1}{3}U_2(t)$ for all $t \in \R$, so $U(\tau) = \frac{1}{3}U_2(\tau)$. It is well known that the eigenvalues of $A(C_9)$ are $\lambda_k = 2\cos{\frac{2\pi k}{9}}$ where $k=0,1,...,8$. Now, we can diagonalize $U(\tau)$ as $U(\tau) = F e^{i\tau D} F^*$ where $D$ is the diagonal matrix of the eigenvalues ($\lambda_k$) of $A(C_9)$, so $3U(\tau) = F 3e^{i\tau D} F^*$. Similarly, we can diagonalize $U_2(\tau) = F D_2 F^*$ where $D_2$ is the diagonal matrix whose diagonal elements are $\mu_k = \sum\limits_{j=0}^8 U_2(\tau)_{1,j}\zeta^{jk}$. Let $\delta_k = 3e^{i\tau\lambda_k}$, where each $\delta_k$ is an eigenvalue of $3U(\tau)$, coming from the transition matrix of $A(C_9)$. From here, we get that $3e^{i\tau D} = D_2 \implies \delta_k = \mu_k$. So for $k = 0$:
\begin{align*}
    \mu_0 &= 1 + u + u\omega^2 + 1 + u\omega + u\omega + 1 + u\omega^2 + u \\
    &= 3 + 2u(1 + \omega + \omega^2) \\
    &= 3, \\
    \delta_0 &= 3e^{i\tau\lambda_0} \\
    &= 3e^{2i\tau}
\end{align*}
Since $\mu_0 = \delta_0,\, 3 = 3e^{2i\tau}$, so $e^{2i\tau} = 1$. It follows that $\tau = n\pi$ for $n \in \Z$. Next, for $k = 1$:
\begin{align*}
    \mu_1 &= 1 + u\zeta + u\omega^2\zeta^2 + \omega + u\omega\zeta^4 + u\omega\zeta^5 + \omega^2 + u\omega^2\zeta^7 + u\zeta^8 \\
    &= (1 + \omega + \omega^2) + u(\zeta + \omega^2\zeta^2 + \omega\zeta^4 + \omega\zeta^5 + \omega^2\zeta^7 + \zeta^8) \\
    &= u(\zeta + \zeta^6\cdot\zeta^2 + \zeta^3\cdot\zeta^4 + \zeta^3\cdot\zeta^5 + \zeta^6\cdot\zeta^7 + \zeta^8) \\
    &= u(\zeta + \zeta^4 + \zeta^7 + 3\zeta^8) \\
    &= u(\zeta(1 + \omega + \omega^2) + 3\zeta^8) \\
    &= 3u\zeta^8, \\
    \delta_1 &= 3e^{i\tau\lambda_1} \\
    &= 3e^{2i\tau\cos(2\pi/9)}
\end{align*}
Now, since $\mu_1 = \delta_1,\, 3u\zeta^8 = 3e^{2i\tau\cos(2\pi/9)}$, so $u\zeta^8 = e^{2i\tau\cos(2\pi/9)}$. Finally, for $k = 2$:
\begin{align*}
    \mu_2 &= 1 + u\zeta^2 + u\omega^2\zeta^4 + \omega^2 + u\omega\zeta^8 + u\omega\zeta + \omega + u\omega^2\zeta^5 + u\zeta^7 \\
    &= (1 + \omega + \omega^2) + u(\zeta^2 + \omega^2\zeta^4 + \omega\zeta^8 + \omega\zeta + \omega^2\zeta^5 + \zeta^7) \\
    &= u(\zeta^2 + \zeta^6\cdot\zeta^4 + \zeta^3\cdot\zeta^8 + \zeta^3\cdot\zeta + \zeta^6\cdot\zeta^5 + \zeta^7) \\
    &= u(\zeta + \zeta^4 + \zeta^7 + 3\zeta^2) \\
    &= u(\zeta(1 + \omega + \omega^2) + 3\zeta^2) \\
    &= 3u\zeta^2, \\
    \delta_2 &= 3e^{i\tau\lambda_2} \\
    &= 3e^{2i\tau\cos(4\pi/9)}
\end{align*}
Similarly, since $\mu_2 = \delta_2$, $3u\zeta^2 = 3e^{2i\tau\cos(4\pi/9)}$, so $u\zeta^2 = e^{2i\tau\cos(4\pi/9)}$. We can thus take the expressions involving $u\zeta^2$ and $u\zeta^8$ and use them to solve for $\tau$.
\begin{align*}
    \frac{u\zeta^8}{u\zeta^2} &= \frac{3e^{2i\tau\cos(4\pi/9)}}{3e^{2i\tau\cos(2\pi/9)}}\\
    \zeta^6 &= e^{2i\tau\big(\cos(4\pi/9) - \cos(2\pi/9)\big)}\\
    e^{4\pi i/3} &= e^{2\sqrt{3}i\tau\sin(\pi/9)}
\end{align*}
Given that $\tau = n\pi$ for $n \in \Z$ from earlier, we can substitute this back into the equation to solve for $n$.
\begin{align*}
    e^{4\pi i/3} &= e^{2\sqrt{3}n\pi i\sin(\pi/9)} \\
    \frac{4\pi}{3} &= 2\sqrt{3}n\pi\sin{\left(\frac{\pi}{9}\right)} + 2m\pi \text{, for } m \in \Z\\
    n &= \frac{2-3m}{3\sqrt{3}\sin{\left(\frac{\pi}{9}\right)}}
\end{align*}

To show $\sqrt{3}\sin(\frac{\pi}{9})$ is an irrational, let us define $x = \sqrt{3}\sin{\left(\frac{\pi}{9}\right)}$.

Using the triple-angle identity \[\sin(3\theta) = 3 \sin(\theta) - 4 \sin^3(\theta)\]

with $\theta = \frac{\pi}{9}$ gives \[\frac{\sqrt{3}}{2} = 3 \sin{\left(\frac{\pi}{9}\right)} - 4 \sin^3{\left(\frac{\pi}{9}\right)}\].

Substituting $\sin(\frac{\pi}{9})$ with $\frac{x}{\sqrt{3}}$, we get 

\begin{align*}
    &\frac{\sqrt{3}}{2} = \sqrt{3}x - \frac{4x^3}{3\sqrt{3}} \\ 
    \implies &9 = 18x - 8x^3 \\
    \implies &8x^3 - 18x + 9 = 0
\end{align*}

By the Rational Root Theorem, any rational root of
\[
8x^3-18x+9
\]
must satisfy $p\mid9$ and $q\mid8$ where $p,q \in \Z$. A direct verification shows that none of the resulting candidates is a root. Therefore, the polynomial has no rational roots. Since $x=\sqrt{3}\sin{\left(\frac{\pi}{9}\right)}$ is a root of this polynomial, it follows that
\[
\sqrt{3}\sin\left(\frac{\pi}{9}\right)
\]
is irrational. Hence, the denominator $3\sqrt{3}\sin{\left(\frac{\pi}{9}\right)}$ is irrational.

Since the numerator is an integer and the denominator is irrational, we can conclude that no $n \in \Z$ exists because $n$ will be an irrational number. Hence, by contradiction, there is no time $t \in \R$ such that uniform mixing occurs for $C_9$.
\end{proof}

\section{Conclusion and Future Work}
In this paper we proved that the cycle graph $C_9$ does not exhibit uniform mixing at any time $t \in \R$. This was the next open problem to explore, as previous authors have already proven mixing results on cycles $C_n$ for $n < 9$. Given the six 2-dimensional families of cyclic 9-roots and 6642 isolated cyclic 9-roots which are known in literature, we were able to easily rule out uniform mixing coming from any of the isolated 9-roots via the machinery in \cite{mullin2013uniform}.

We then inspected the six 2-dimensional families. Our techniques involved generalizing these families and converting them into a cyclic Type-II matrix. By doing so, we were able to compare the structure of this matrix with the Fourier-diagonalized transition matrix on $C_9$ and determine the necessary algebraic conditions necessary for uniform mixing to occur.

These constraints forced the condition that the time $t$ at which uniform mixing occurs must be an integer multiple of $2\pi$. However, later calculations showed that this multiple of $2\pi$ at which uniform mixing occurs must be irrational number, hence contradicting the requirement that it be an integer.

Therefore, we were able to show that there is no choice of cyclic 9-root --- neither isolated nor 2-dimensional --- that leads to uniform mixing on $C_9$, and thus conclude that the cycle $C_9$ never exhibits uniform mixing at any time $t \in \R$.

Given that uniform mixing results have already been proven for even cycles and odd prime cycles, we identify $C_{21}$ as the next open case. Our techniques of constructing and utilizing Gröbner bases of cyclic $n$-roots prove to be too computationally heavy for these larger values of $n$. Consequently, we do not employ these techniques for the $C_{21}$ case and recognize that other methods are needed for the proof.

\section{Acknowledgements}

We would like to thank our supervisors Dr.~Nathaniel Johnston and Dr.~Sarah Plosker for their mentorship, guidance, and continual encouragement throughout this research process. We are grateful for the time they both invested to discuss ideas, answer questions, and provide thoughtful feedback that greatly improved the quality of this paper. A.G.~gratefully acknowledges financial support by Dr.~Sarah Plosker's NSERC Discovery Grant number RGPIN-2025-05704. P.P.~was  supported by the Natural Sciences and Engineering Research Council of Canada (NSERC)  Undergraduate Student Research Award (USRA) Application number USRA-615931-2026. I.Y.~was supported by Dr.~Nathaniel Johnston's NSERC Discovery Grant number RGPIN-2022-04098.
\bibliographystyle{alpha}
\bibliography{ref}
\end{document}